\documentclass[a4paper]{article}

\usepackage[left=1.5cm,top=2.5cm,right=1.5cm,bottom=2.5cm]{geometry}

\title{New composition products for complex harmonic functions, the dynamic with respect to these and composition operators induced}
\author{ {\bf{Luis E. Ben\'{\i}tez$^{a}$} and \bf{Ra\'ul Felipe$^{b}$}}\\ \\
        $^{a}$ Departamento de Matem\'{a}ticas y Estad\'{\i}stica \\
        Universidad de C\'{o}rdoba \\
        Carrera $6$, No. $76-103$, Monter\'{\i}a, Colombia \\
        luisbenitezb@correo.unicordoba.edu.co \\
        $^{b}$ CIMAT \\
        Callej\'on Jalisco s/n Mineral de Valenciana \\
        Guanajuato, Gto., M\'exico.\\
        raulf@cimat.mx}
\date{}
\usepackage{graphicx,array,delarray}
\usepackage{latexsym}
\usepackage{amscd}
\usepackage{amsfonts, amsmath, amssymb}

\newcommand{\CC}{\mathbb{C}}

\newtheorem{theorem}{\textbf{Theorem}}

\newtheorem{corollary}[theorem]{\textbf{Corollary}}

\newtheorem{example}{\textbf{Example}}
\newtheorem{definition}[theorem]{\textbf{Definition}}

\newtheorem{lemma}[theorem]{\textbf{Lemma}}

\newtheorem{proposition}[theorem]{\textbf{Proposition}}

\newtheorem{remark}[theorem]{Remark}

\newenvironment{proof}[1][Proof]{\noindent\textbf{#1.} }{\ \rule{0.5em}{0.5em}}

\begin{document}
\maketitle

\begin{abstract}
In this work we propose composition products in the class of complex harmonic functions so that the composition of two such functions is
again a complex harmonic function. From here we begin the study of the iterations of the functions of this class showing briefly their potential
to be a topic of future research. In parallel, we define and study composition operators whose symbols belong to a Hardy space of complex harmonic functions also introduced in the work. All this constitutes a previous work for the research of semigroups and evolutionary families composed of complex
harmonic functions.
\end{abstract}

\noindent{\it{2020 Mathematics Subject Classification (MSC2020):}}
Primary 31A05; Secondary 30H10.\\

\noindent{\it{Key words:}}\,\,\,\,Complex harmonic functions, complex harmonic dynamics, composition operators.\\

\section{Introduction}

Unlike what happens in the class of analytical functions which is closed under the customary composition, the usual composition
product of two harmonic functions is not in general a harmonic function. This fact causes that some problems which are studied for
a long time in the space of analytical functions do not make sense or are difficult to translate and treat on the set of complex
harmonic functions with the tools of the complex variable. We give two typical examples: the theory of linear composition operator
whose symbols are complex harmonic functions and the corresponding theory of iterations for complex harmonic functions.

In this note, we will introduce two composition products in the space complex harmonic function in order to overcome these difficulties.
From these products, in this work we begin the study of the problems before mentioned in this class of complex harmonic functions.

A continuous function $f=u+iv$ defined in a domain $D\subset \CC$ is harmonic in $D$ if $u$ and $v$
are real harmonic functions in $D$ which are not necessarily conjugate. If $D$ is simply connected domain we can write $f=h+\overline{g}$,
where $h$ and $g$ are analytic functions on $D$ and $\overline{g}$ denotes the function $z\longrightarrow \overline{g(z)}$.
In this paper, we should only consider complex harmonic functions $f=h+\overline{g}$ for which $g\neq 0$.
From now on, all plane domain $D$ must be considered simply connected.

It is clear that we can consider $f$ as a function of two variable $z$ and $\overline{z}$, in this case $\partial_{z}f(z)=\partial_{z}h(z)$
and $\partial_{\overline{z}}f(z)=\overline{\partial_{z}g(z)}=\partial_{\overline{z}}\overline{g}(\overline{z})$. Let us denote by $H(D)$
the set of all complex harmonic functions defined on $D$. Obviously $f\in H(D)$ if and only if $\partial^{2}_{z\overline{z}}f=0$. Observe also
that if for $f=h+\overline{g}\in H(D)$ we have $g=0$ then $f$ is analytic, in this case, $u$ and $v$ satisfy the Cauchy-Riemann equations.
We can denote by $A(D)$ the set of all analytic functions on $D$. For $f=h+\overline{g}\in H(D)$ we say that $h$ is its analytic part
while that $g$ will be called the co-analytic part. In the opinion of the authors the basic reference for the study of complex harmonic
functions is \cite{duren}.

The theory of iterations of a function $f(z)$ of the complex variable $z$ studies the sequence the `iterates' $\{f_{n}(z)\}$ defined by
\begin{equation*}
f_{0}(z)=z,\,\,f_{1}(z)=f(z),\,\,\,f_{n+1}(z)=f_{1}(f_{n}(z)),\,\,\,\,\,\,\,\,\,\,\,\,\,n=0,1,2,\cdots .
\end{equation*}

In the general theory of iterations, the fixed points of $f(z)$ play a fundamental role both in local theories and others which deal with the behavior
of the sequence $\{f_{n}(z)\}$ and the solutions of some functional equations in the neighbourhood of fixed points. In this work we begin
the theory of iteration of complex harmonic functions.

On the other hand, an essential objective of the theory of linear composition operators is to obtain information from the operators from
the properties of the function class to which their symbols belong. It is very convenient, for example, that this class of functions has
a Banach or Hilbert space structure (this space usually coincides with the class on which the operators act).

This work can be considered a first step in the construction of semigroups and evolutionary families in the space of complex harmonic functions.
Indeed, it is well know the fundamental role of the composition operators and the conjugation techniques initially arise in the theory of iterations for analytical
functions within this topic in the class $A(D)$, research which we wish to undertake in a forthcoming paper to the class $H(D)$. This program also includes inserting the
theory of integrable systems in this context.

\section{Dynamic of complex harmonic functions}

We introduce a first composition law between harmonic functions for which the function $z+\overline{z}\in H(\CC)$ should be the identity
function.

\begin{definition}Let $f_{1}=h_{1}+\overline{g_{1}}$ and  $f_{2}=h_{2}+\overline{g_{2}}$ be two functions of $H(D_{1})$
and $H(D_{2})$ respectively, such that, $h_{2}(D_{2})\subset D_{1}$ and $g_{2}(D_{2})\subset D_{1}$. We define a harmonic composition product
of $f_{1}$ with $f_{2}$ which is denoted by $f_{1}\circleddash f_{2}$ in the following form
\begin{equation}\label{rf1}
f_{1}\circleddash f_{2}(z)=h_{1}\circ h_{2}(z)+\overline{g_{1}\circ
g_{2}(z)},\,\,\,\,\,\hbox{for all $z\in D_{2}$},
\end{equation}
where $\circ$ denotes the usual composition between analytic functions. The composition defined by (\ref{rf1}) will be
called \textbf{direct harmonic composition}.
\end{definition}

Let $H_{s}(D)$ be the subset of $H(D)$ consists of whose functions $f=h+\overline{g}$ for which $h:D\longrightarrow D$ and
$g:D\longrightarrow D$. Note that $z+\overline{z}\in H_{s}(D)$.

Our first result is
\begin{proposition}$(H_{s}(D),\circleddash)$ is a unital semigroup and the function $z+\overline{z}$ is its unit.
\end{proposition}
\begin{proof}The proof follows taking into account that the usual composition for analytic functions is an associative operation. On other
hand, it is trivial to verify that $z+\overline{z}$ is the unit with respect to $\circleddash$.
\end{proof}

Although for a $f\in H_{s}(D)$ fixed, we can not guarantee in general that $f(D)\subset D$, we call to the elements of $H_{s}(D)$
harmonic $\circleddash$-automorphisms of $D$ with respect to the composition $\circleddash$.

In this paper, a univalent harmonic function is a function $f=h+\overline{g}\in H_{s}(D)$ for which both $h$ and $g$ are univalent analytic functions.
Then $f^{-1}=h^{-1}+\overline{g^{-1}}\in H_{s}(D)$ and $f\circleddash f^{-1}=z+\overline{z}=f^{-1}\circleddash f$.

We introduce a second composition product between harmonic maps.

\begin{definition}Let $f_{1}=h_{1}+\overline{g_{1}}$ and  $f_{2}=h_{2}+\overline{g_{2}}$ be two functions of $H(D_{1})$
and $H(D_{2})$ respectively, such that, $h_{2}(D_{2})\subset D_{1}$ and $g_{2}(D_{2})\subset D_{1}$. We define other harmonic composition
$f_{1}\circledcirc f_{2}$ between two functions $f_{1}$ and $f_{2}$ in the form
\begin{equation}\label{rf11}
f_{1}\circledcirc f_{2}(z)=h_{1}\circ g_{2}(z)+\overline{g_{1}\circ
h_{2}(z)},\,\,\,\,\,\hbox{for all $z\in D_{2}$},
\end{equation}
where as before $\circ$ denotes the usual composition of analytic functions. Next, this composition should be
called the \textbf{crossed harmonic composition}. Here, the function $z+\overline{z}$ is the identity.
\end{definition}

Notice that the composition $\circledcirc$ is not associative.

\begin{example}A M\"{o}bius harmonic transformation is a harmonic function of the form
\begin{equation*}
R(z)=\frac{az+b}{cz+d}+\overline{\frac{lz+n}{sz+t}}=T_{A}(z)+\overline{T_{B}(z)},
\end{equation*}
where
\begin{equation*}
A=\left(
    \begin{array}{cc}
      a & b \\
      c & d \\
    \end{array}
  \right),\,\,\,\,B=\left(
                      \begin{array}{cc}
                        l & n \\
                        s & t \\
                      \end{array}
                    \right)
\end{equation*}
are matrices of $GL_{2}(\mathbb{C})$.
Let $R_{1}(z)=T_{A_{1}}(z)+\overline{T_{B_{1}}(z)}$ and $R_{2}(z)=T_{A_{2}}(z)+\overline{T_{B_{2}}(z)}$ be two M\"{o}bius harmonic transformations, then
\begin{equation}\label{rf15}
R_{1}(z)\circleddash R_{2}(z)=T_{A_{1}}(z)\circ T_{A_{2}}(z)+\overline{T_{B_{1}}(z)\circ T_{B_{2}}(z)}=T_{A_{1}A_{2}}(z)+\overline{T_{B_{1}B_{2}}(z)},
\end{equation}
and
\begin{equation}\label{rf16}
R_{1}(z)\circledcirc R_{2}(z)=T_{A_{1}}(z)\circ T_{B_{2}}(z)+\overline{T_{B_{1}}(z)\circ T_{A_{2}}(z)}=T_{A_{1}B_{2}}(z)+\overline{T_{B_{1}A_{2}}(z)}.
\end{equation}
\end{example}

\begin{definition}
We say that a function $f_{1}\in H_{s}(D)$ is conjugated with respect to the harmonic composition $\circ_{H}$ to a function $f_{2}\in H_{s}(D)$
if there exists a univalent harmonic function $f\in H_{s}(D)$ such that
\begin{equation}\label{rf2}
f\circ_{H}f_{1}=f_{2}\circ_{H}f,
\end{equation}
where $\circ_{H}\in\{\circleddash,\circledcirc\}$.
\end{definition}

In what follows, if $f=h+\overline{g}$ then the function $\overline{f}=g+\overline{h}$ will be called the conjugate complex harmonic function of $f$.
Let $f=h+\overline{g}\in H_{s}(D)$ given, we define $f^{n,\circleddash}=f\circleddash f^{n-1,\circleddash}$ and
$f^{n,\circledcirc}=f^{n-1,\circledcirc}\circledcirc f$ for $n\geq 2$, where $f^{1,\circleddash}=f^{1,\circledcirc}=f$
and $f^{0,\circleddash}(z)=z+\overline{z}$. Then
\begin{equation}\label{rf5}
f^{k,\circleddash}(z)=h^{k}(z)+\overline{g^{k}(z)},\,\,\,\,\,\,f^{k,\circledcirc}(z)=h(g^{k-1}(z))+\overline{g(h^{k-1}(z))}
=(h+\overline{g})\circledcirc f^{k-1,\circleddash}(z)=f\circledcirc f^{k-1,\circleddash}(z),
\end{equation}
for all $z\in D$ (or what is the same for all $z+\overline{z}\in D+\overline{D}$) and $1\leq k$. We use $h^{k}$ to represent usual
composed with itself $k$ times of the function $f$ and the same for $g^{k}$, for  all $k$. Also, here and throughout the section $\overline{D}=\{\overline{z}|\,z\in D\}$.

We need to introduce the subset of constant complex harmonic functions: $\mathbb{C}_{H}=\mathbb{C}+\overline{\mathbb{C}}$ and call its elements
the harmonic complex numbers. Suppose that $f=h+\overline{g}\in H_{s}(D)$, we say that the sequence $\{f^{n,\circleddash}(z)\}$ is
convergent in the point $z+\overline{z}\in D+\overline{D}$, if by definition, there exits a harmonic complex number $\mu+\overline{\omega}$ such that $f^{n,\circleddash}(z)\longrightarrow \mu+\overline{\omega}$.
It is easy to see that if $\{f^{n,\circleddash}(z)\}$
is convergent to $\mu+\overline{\omega}$ in $z+\overline{z}$ (or $z\in D$), then $\{(\overline{f})^{n,\circleddash}(z)\}$ is convergent to $\omega+\overline{\mu}$ in $z+\overline{z}\in D+\overline{D}$ and vice versa.

\begin{example}Assume that $f=h+\overline{g}\in H_{s}(D)$ is such that $\{h^{n}(z)\}$ converges to a finite point $\mu$ and
$\{g^{n}(z)\}$ also converges to a finite point $\omega$ for some $z\in D$, clearly $\mu, \omega\in cl(D)$, where $cl(D)$ stands for the closure of $D$.
It is well known that $\mu$
and $\omega$ should be fixed points of $h$ and $g$ respectively. Then, the sequence $\{f^{n,\circleddash}(z)\}$ converges to
$\mu+\overline{\omega}$, while $\{f^{n,\circledcirc}(z)\}$ is convergent to $h(\omega)+\overline{g(\mu)}$.
\end{example}

We give the following useful general result

\begin{proposition}Let us suppose that $f=h+\overline{g}\in H_{s}(D)$, then we have
\begin{itemize}
  \item if for $z\in D$ (or same if for $z+\overline{z}\in
  D+\overline{D}$) the sequence $\{f^{n,\circleddash}(z)\}$ is convergent to $\mu+\overline{\omega}$, then $\mu+\overline{\omega}=h(\mu)+\overline{g(\omega)}$,
  \item assuming that $\{f^{n,\circleddash}(z)\}$ is convergent to $\mu+\overline{\omega}$ in a point $z\in D$ given, then $\{f^{n,\circledcirc}(z)\}$
  converges to $h(\omega)+\overline{g(\mu)}$ in $z$. If $f$ is univalent then $\{f^{n,\circledcirc}(z)\}$ converges in $z\in D$ if and only if $\{f^{n, \circleddash}(z)\}$ converges in this point.
\end{itemize}
\end{proposition}

\begin{proof}From the first equality of (\ref{rf5}) it follows that $f^{n+1,\circleddash}(z)=f\circleddash f^{n,\circleddash}(z)$ for all $n\in \mathbb{N}$, then passing to the limit on both sides, we obtain $\mu +\overline{\omega}=f\circleddash (\mu +\overline{\omega})=h(\mu)+\overline{g(\omega)}$ which proves the first statement. Also, for the second equality of (\ref{rf5}) one has $f^{n,\circledcirc}(z)=f\circledcirc f^{n-1,\circleddash}(z)$ for all $1\leq n$ and for our assumptions $\{f^{n,\circleddash}(z)\}$ is convergent to $\mu+\overline{\omega}$ which implies that $\{f^{n,\circledcirc}(z)\}$ converges to $(h+\overline{g})\circledcirc (\mu+\overline{\omega})=
h(\omega)+\overline{g(\mu)}$. Finally, observe that if $f=h+\overline{g}$ is univalent, then further we have $(h^{-1}+\overline{g^{-1}})\circleddash f^{n,\circledcirc}(z)=(\overline{f})^{n,\circleddash}(z)$. It shows that $\{f^{n,\circledcirc}(z)\}$ converges if and only if
$\{f^{n, \circleddash}(z)\}$ converges.
\end{proof}

\begin{definition}The complex harmonic constant $\mu+\overline{\omega}$ is said to be a finite $\mathfrak{h}$-fixed point for the complex harmonic function $f=h+\overline{g}\in H_{s}(D)$ if this satisfies the equation $\mu+\overline{\omega}=h(\mu)+\overline{g(\omega)}$.
\end{definition}

Suppose that $f=h+\overline{g}\in H_{s}(D)$ such that $h(\mu)=\mu$ and $g(\omega)=\omega$, then $\mu+\overline{\omega}=h(\mu)+\overline{g(\omega)}$.
Thus $\mu+\overline{\omega}$ is a $\mathfrak{h}$-fixed point of $f$. These types of points will be called induced $\mathfrak{h}$-fixed points and
they constitute probably isolated $\mathfrak{h}$-fixed points. In particular, if $g(0)=0$ then all the usual fixed points of $h$ are
$\mathfrak{h}$-fixed points of $f=h+\overline{g}$. In the same way, if $h(0)=0$ then each fixed point of $g$ is $\mathfrak{h}$-fixed point of $f=h+\overline{g}$. It shows that $0$ is a $\mathfrak{h}$-fixed point of $f=h+\overline{g}$ whenever $h(0)=g(0)=0$.

In the case $D=\mathbb{C}$ is reasonable to include $z=\infty$ in the analysis of the $\mathfrak{h}$-fixed points of a complex harmonic functions. We say that
\begin{enumerate}
  \item $\mu+\overline{\infty}$ is an infinite $\mathfrak{h}$-fixed point of $f=h+\overline{g}$ if $h(\mu)=\mu$ and $g(\infty)=\infty$,
  \item $\infty+\overline{\omega}$ is an infinite $\mathfrak{h}$-fixed point of $f=h+\overline{g}$ if $h(\infty)=\infty$ and $g(\omega)=\omega$,
  \item $\infty+\overline{\infty}$ is an infinite $\mathfrak{h}$-fixed point of $f=h+\overline{g}$ if $h(\infty)=\infty$ and $g(\infty)=\infty$.
\end{enumerate}

For example, for a M\"{o}bius harmonic function $R(z)=T_{A}(z)+\overline{T_{B}(z)}$, we can have up to $4$ $\mathfrak{h}$-fixed points. There are several possibilities which could be showed and analyzed from the point of view of the convergence of the sequence $\{R^{n,\circleddash}(z)\}$.

\begin{itemize}
\item The M\"{o}bius functions $T_{A}(z)$ and $T_{A}(z)$ have a single fixed point:

\begin{itemize}
  \item Both M\"{o}bius functions $T_{A}(z)$ and $T_{B}(z)$ have $\infty$ as their only fixed point, that is, $\infty+\overline{\infty}$ is the unique $\mathfrak{h}$-fixed point. Then, $R(z)=z+\overline{z}$ or
  $R(z)=R_{\beta}(z)=z+\beta +\overline{z}$ where $\beta\neq 0$. Hence, $R^{n,\circleddash}_{\beta}(z)\longrightarrow \infty +\overline{z}=\infty$ for all $z\in \mathbb{C}$,
  \item The sets of fixed points of $T_{A}(z)$ and $T_{B}(z)$ are $FP_{A}=\{\mu\}$ and $FP_{B}=\{\infty\}$ respectively. In this case, one has $R^{n,\circleddash}(z)\longrightarrow \mu+\overline{z}$
  or $R^{n,\circleddash}(z)\longrightarrow \mu+\overline{\infty}=\infty$ for all $z\in\mathbb{C}$,
  \item $FP_{A}=\{\infty\}$ and $FP_{B}=\{\omega\}$, then $R^{n,\circleddash}(z)\longrightarrow z+\overline{\omega}$ or $R^{n,\circleddash}(z)\longrightarrow \infty+\overline{\omega}=\infty$, for all
  $z\in\mathbb{C}$,
  \item For $FP_{A}=\{\mu\}$ and $FP_{B}=\{\omega\}$, we find that $R^{n,\circleddash}(z)\longrightarrow \mu+\overline{\omega}$, for all $z\in \mathbb{C}$.
\end{itemize}

\item In the remaining cases, $R^{n,\circleddash}(z)$ converges for all $z\in\mathbb{C}$ if and only if $\infty$ does not belong to $FP_{A}\cup FP_{B}$
and there are $\mu\in FP_{A}$ and $\omega\in FP_{B}$ such that $h^{n}(z)\longrightarrow \mu$ and $g^{n}(z)\longrightarrow \omega$.
\end{itemize}

For the remainder of this section, we will only work with the composition product $\circleddash$.

\begin{theorem}Let $\mu+\overline{\omega}\in D$ be a $\mathfrak{h}$-fixed point of $f=h+\overline{g}\in H(D)$.
Assume that there exist a neighborhood $V\subset D$ of $\mu+\overline{\omega}$ and two positive constants $\rho_{h}<1$ and $\rho_{g}<1$ such that $h(V)\subset V$, $g(V)\subset V$ and the following inequalities
\begin{equation}\label{rf6}
|h(z)-h(\mu)|<\rho_{h}|z-h(\mu)|,
\end{equation}
\begin{equation}\label{rf7}
|g(z)-g(\omega)|<\rho_{g}|z-g(\omega)|,
\end{equation}
hold for all $z\in V$. Then sequence $\{f^{n,\circleddash}\}_{n\geq 1}$ converges
uniformly to $\mu+\overline{\omega}$ in $V$. In this case, we say that $\mu+\overline{\omega}$ is an
\textbf{attracting $\mathfrak{h}$-fixed point}.
\end{theorem}
\begin{proof}From (\ref{rf6}) and (\ref{rf7}) follow that for all $n$ with $2\leq n$ we obtain
\begin{equation}\label{rf8}
|h^{n}(z)-h(\mu)|<(\rho_{h})^{n}|z-h(\mu)|,\,\,\,\,\,\hbox{for all $z\in V$},
\end{equation}
and
\begin{equation}\label{rf9}
|g^{n}(z)-g(\omega)|<(\rho_{g})^{n}|z-g(\omega)|,\,\,\,\,\,\,\,\,\hbox{for any $z\in V$}.
\end{equation}

Then using (\ref{rf8}) and (\ref{rf9}), we have
\begin{align}
|f^{n,\circleddash}(z)-(\mu+\overline{\omega})|&=|(h^{n}(z)-h(\mu))+\overline{(g^{n}(z)-g(\omega))}| \nonumber \\
&<|h^{n}(z)-h(\mu)|+|g^{n}(z)-g(\omega)| \nonumber  \\
&<(\rho_{h})^{n}|z-h(\mu)|+(\rho_{g})^{n}|z-g(\omega)|, \nonumber
\end{align}
consequently, the sequence of harmonic iterates $\{f^{n,\circleddash}\}$ converges uniformly to $\mu+\overline{\omega}$ in $V$.
\end{proof}

\begin{example}Assume now that $U=\{z\in\CC||z|<1\}$ and let us choose arbitrary $\alpha$ and $\beta$ in $U$. Then
$z_{0}=0$ is a attracting $\mathfrak{h}$-fixed point of $f(z)=\alpha z+\beta \overline{z}$.
\end{example}

Now, we present a generalization of a result due to Koenigs (see \cite{carl}, page $31$). Let $\mu+\overline{\omega}$ be a $\mathfrak{h}$-fixed point of $f=h+\overline{g}\in H_{s}(D)$ then
$\lambda=\partial_{z} h(\mu)$ and $\theta=\partial_{z}g(\omega)$ are called the multipliers of $f$.

\begin{theorem}\label{rf11}
Assume that $f=h+\overline{g}\in H_{s}(U)$ has a fixed point in
$z_{0}=0$ (this means that $h(0)=0=g(0)$) with multipliers $\lambda$ and
$\theta$ satisfying $0<|\lambda| <1$, $0<|\theta| <1$. Then there exists a
neighborhood $V\subset U$ of zero and
$\varphi=\varphi_{h}+\overline{\varphi_{g}} \in H(V)$ such that
\begin{equation}\label{rf10}
\varphi \circleddash f=\lambda \varphi_{h}+\overline{\theta
\varphi_{g}}=(\lambda z+\overline{\theta z})\circleddash \varphi,
\end{equation}
and $\varphi(0)=0$.
\end{theorem}
\begin{proof}By a well known Theorem due to Koenigs (1884) there are two
neighborhoods of zero $V_{h}\subset U$ and $V_{g}\subset U$ and
two functions $\varphi_{h}$ and $\varphi_{g}$ which are analytic in
$V_{h}$ and $V_{g}$ respectively, such that, $\varphi_{h}\circ
h(z)=\lambda \varphi_{h}(z)$ for all $z\in V_{h}$ and
$\varphi_{g}\circ g(z)=\theta \varphi_{g}(z)$ for any $z\in V_{g}$, where $\lambda=\partial_{z} h(0)$ and $\theta=\partial_{z}g(0)$.
Observe that necessarily must be $\varphi_{h}(0)=0=\varphi_{g}(0)$.
Let us define $V=V_{h}\cap V_{g}$ then
$\varphi(z)=\varphi_{h}(z)+\overline{\varphi_{g}(z)}\in H(V)$ and
satisfies (\ref{rf10}).
\end{proof}

If $f$ satisfies the hypothesis of the previous Theorem we must add
that $z_{0}=0$ is also an attracting $\mathfrak{h}$-fixed point since evidently the
conditions (\ref{rf6}) and (\ref{rf7}) hold.

\begin{proposition}Under the hypothesis of the Theorem \ref{rf11}
the sequence of iterates $\{f^{n,\circleddash}\}$ converges uniformly to $z_{0}=0$ on some
neighborhood $V_{0}$.
\end{proposition}
\begin{proof}It is clear that $z_{0}=0$ is an ordinary attracting fixed point
of $h$ and $g$. Hence, the iterates $h^{n}$ and $g^{n}$ constitute two sequences which converge
uniformly to $0$ on neighborhoods $V_{1}$ and $V_{2}$
respectively. Since, $f^{n,\circleddash}=h^{n}+\overline{g^{n}}$ it shows that the
sequence if iterates $\{f^{n,\circleddash}\}$ converges uniformly to $z_{0}=0$ in
$V_{0}=V_{1}\cap V_{2}$.
\end{proof}

Let us assume now that $z_{0}=0$ is a $\mathfrak{h}$-fixed point of
$f=h+\overline{g}\in H_{s}(U)$ such that $\lambda=0$ and $0<\theta <1$.
We recall that $h(0)=0=g(0)$, and consider $h(z)$ of the form
\begin{equation}\label{rf12}
h(z)=a_{p}z^{p}+\cdots,\,\,\,\,a_{p}\neq 0,\,\,p\geq 2,
\end{equation}
then by a well known result of Boettcher (1904) (the result can be found in \cite{carl}, page $33$) there exists a
neighborhood of zero $V_{h}$ and an analytic function $\varphi_{h}$
over it such that
\begin{equation}\label{rf13}
\varphi_{h}\circ h(z)=(\varphi (z))^{p},\,\,\,\,\hbox{for any $z\in V_{h}$}.
\end{equation}

Hence, we can enunciate the following result
\begin{theorem}
Assume that $f=h+\overline{g}\in H_{s}(U)$ has a $\mathfrak{h}$-fixed point in $z_{0}=0$ (that is, $h(0)=0=g(0)$) with
multipliers $\lambda$ and $\theta$ satisfying $\lambda=0$, $0<|\theta| <1$. Suppose we are given $h$ by (\ref{rf12}),
then there exists a neighborhood $V$ of zero and
$\varphi=\varphi_{h}+\overline{\varphi_{g}} \in H(V)$ such that
\begin{equation}\label{rf14}
\varphi \circleddash f=(\varphi_{h})^{p}+\overline{\theta
\varphi_{g}}=(z^{p}+\overline{\theta z})\circleddash \varphi,
\end{equation}
where $\varphi_{h}(0)=0$ or $\varphi_{h}(0)$ is a $(p-1)$ th root of unity.
\end{theorem}

Analogous results can be proved in the cases when $0<|\lambda|<1, \theta=0$ and $\lambda=0=\theta$.

\section{Composition operators on the space $H^{2}(\mathbb{D})+\overline{H^{2}(\mathbb{D})}$}

It is the object of this section to investigate the theory of composition operators in the framework of the complex harmonic functions.
The section is organized as follows, first of all we introduce an analogue of Hardy space in the class of complex harmonic functions and give some results for the linear operators
defined in this space, in particular we develop with some depth the theory of linear composition operators whose symbols are complex harmonic functions. We conclude the section
showing the relationship of these composition operators with the corresponding complex harmonic reproducing kernels.

\subsection{The Hardy type space $H^{2}(\mathbb{D})+\overline{H^{2}(\mathbb{D})}$ and its corresponding space $\mathcal{B}(H^{2}(\mathbb{D})+\overline{H^{2}(\mathbb{D})})$}

We recall that the separable Hilbert space $H^{2}(\mathbb{D})$ consists of all analytic functions having power series representations with
square-summable complex coefficients, in other words
$$H^{2}(\mathbb{D})=\left \{f \,| f(z)=\sum_{n=0}^{\infty}a_{n}z^{n},\,\,\sum_{n=0}^{\infty}|a_{n}|^{2}< \infty \right \}.$$

Any function of $H^{2}(\mathbb{D})$ is analytic in the open unit disc $\mathbb{D}$. The inner product on $H^{2}$ is defined to be $\langle f, g\rangle=\sum_{n=0}^{\infty}a_{n}\overline{b_{n}}$
for $f=\sum_{n=0}^{\infty}a_{n}z^{n}$ and $g=\sum_{n=0}^{\infty}b_{n}z^{n}$. Thus, the norm of the vector $f=\sum_{n=0}^{\infty}a_{n}z^{n}$ is $\|f\|=\left ( \sum_{n=0}^{\infty}|a_{n}|^{2}\right )^{\frac{1}{2}}$.
Each analytic function $\phi$ mapping the unit disc into itself defines a composition operator $C_{\phi}: H^{2}(\mathbb{D})\longrightarrow H^{2}(\mathbb{D})$ of the form $(C_{\phi}f)(z)=f(\phi(z))$ for all $f\in H^{2}(\mathbb{D})$
and $C_{\phi}\in \mathcal{B}(H^{2}(\mathbb{D}))$ where $\mathcal{B}(H^{2}(\mathbb{D}))$ is the space of all bounded linear operators of $H^{2}(\mathbb{D})$ into
$H^{2}(\mathbb{D})$. The function $\phi$ is called the symbol of $C_{\phi}$. Moreover, we recall that
$$\|C_{\phi}\|\leq \sqrt{\frac{1+|\phi(0)|}{1-|\phi(0)}},$$
in particular if $\phi(0)=0$ then $\|C_{\phi}\|=1$.

We use the symbol $HH^{2}(\mathbb{D})$ to denote the space of harmonic complex functions $f+\overline{g}$ where $f,g\in H^{2}(\mathbb{D})$.
It is clear that
$HH^{2}(\mathbb{D})$ is a vector spaces. Moreover, it is possible to introduced in $HH^{2}(\mathbb{D})$ an inner product, more exactly
one has

\begin{lemma}$HH^{2}(\mathbb{D})$ is a Hilbert space with respect to the following inner product
\begin{equation}\label{s1-1}
(a+\overline{b},c+\overline{d})_{HH^{2}(\mathbb{D})}=(a,c)_{H^{2}(\mathbb{D})}+\overline{(b,d)_{H^{2}(\mathbb{D})}}=
(a,c)_{H^{2}(\mathbb{D})}+(d,b)_{H^{2}(\mathbb{D})},
\end{equation}
for all $a+\overline{b}, c+\overline{d}\in HH^{2}(\mathbb{D})$.
\end{lemma}
\begin{proof}Indeed,
\begin{equation*}
(a+\overline{b},a+\overline{b})_{HH^{2}(\mathbb{D})}=\parallel a \parallel^{2}_{H^{2}(\mathbb{D})}+\parallel b \parallel^{2}_{H^{2}(\mathbb{D})}\geq 0,
\end{equation*}
and $(a+\overline{b},a+\overline{b})_{HH^{2}(\mathbb{D})}=0$ if and only if $\parallel a \parallel^{2}_{H^{2}(\mathbb{D})}=0$ and
$\parallel b \parallel^{2}_{H^{2}(\mathbb{D})}=0$, that is if $a+\overline{b}=0$. Now
\begin{align*}
(\alpha_{1}(a_{1}+\overline{b_{1}})+\alpha_{2}(a_{2}+\overline{b_{2}}),c+\overline{d})_{HH^{2}(\mathbb{D})}&=
(\alpha_{1}(a_{1}+\overline{b_{1}}),c+\overline{d})_{HH^{2}(\mathbb{D})}+(\alpha_{2}(a_{2}+\overline{b_{2}}),c+\overline{d})_{HH^{2}(\mathbb{D})} \\
&=(\alpha_{1}a_{1}+\alpha_{2}a_{2},c)_{H^{2}(\mathbb{D})}+(d,\overline{\alpha_{1}}b_{1}+\overline{\alpha_{2}}b_{2})_{H^{2}(\mathbb{D})} \\
&=\alpha_{1}\left [(a_{1},c)_{H^{2}(\mathbb{D})} +(d,b_{1})_{H^{2}(\mathbb{D})} \right ]+
\alpha_{2}\left [(a_{2},c)_{H^{2}(\mathbb{D})} +(d,b_{2})_{H^{2}(\mathbb{D})} \right ] \\
&=\alpha_{1}(a_{1}+\overline{b_{1}},c+\overline{d})_{H^{2}(\mathbb{D})}+\alpha_{2}(a_{2}+\overline{b_{2}},c+\overline{d})_{H^{2}(\mathbb{D})}.
\end{align*}

Also,
\begin{equation*}
(a+\overline{b},c+\overline{d})_{HH^{2}(\mathbb{D})}=(a,c)_{H^{2}(\mathbb{D})}+(d,b)_{H^{2}(\mathbb{D})}=\overline{(c,a)_{H^{2}(\mathbb{D})}}
+\overline{(b,d)_{H^{2}(\mathbb{D})}}=\overline{(c+\overline{d},a+\overline{b})_{HH^{2}(\mathbb{D})}}\,\,.
\end{equation*}

Finally, if $\{a_{n}+\overline{b_{n}}\}_{n\geq 0}$ is a sequence of Cauchy in $HH^{2}(\mathbb{D})$, then $\{a_{n}\}_{n\geq 0}$ and $\{b_{n}\}_{n\geq 0}$
are of Cauchy in $H^{2}(\mathbb{D})$. It shows that the sequence $\{a_{n}+\overline{b_{n}}\}_{n\geq 0}$ is convergent, thus $HH^{2}(\mathbb{D})$ is complete.
\end{proof}

\begin{lemma}The functions $\{e_{n}, f_{n}\}_{n\geq 0}$, where $e_{n}=z^{n}$ and $f_{n}=\overline{z^{n}}$ for all $n\in \mathbb{N}$ form an orthonormal basis of $HH^{2}(\mathbb{D})$.
\end{lemma}
\begin{proof}In fact
\begin{equation*}
(e_{n},f_{m})_{HH^{2}(\mathbb{D})}=(z^{n},\overline{z^{m}})_{HH^{2}(\mathbb{D})}=(z^{n},0)_{H^{2}(\mathbb{D})}+\overline{(0,z^{m})_{H^{2}(\mathbb{D})}}=0,
\end{equation*}
if $n,m\in\mathbb{N}\backslash\{0\}$. Now, $(e_{n},e_{m})_{HH^{2}(\mathbb{D})}=(z^{n},z^{m})_{H^{2}(\mathbb{D})}=\delta_{nm}$ and
$(f_{n},f_{m})_{HH^{2}(\mathbb{D})}=\overline{(z^{n},z^{m})_{H^{2}(\mathbb{D})}}=\delta_{nm}$. On the other hand, being $a+\overline{b}=\sum a_{n}z^{n}+\overline{\sum b_{n}z^{n}}$
then $(a+\overline{b}, e_{n}(z))_{HH^{2}(\mathbb{D})}=(a,z^{n})_{H^{2}(\mathbb{D})}+\overline{(b,0)_{H^{2}(\mathbb{D})}}=a_{n}$, and
$(a+\overline{b}, f_{n}(z))_{HH^{2}(\mathbb{D})}=(a,0)_{H^{2}(\mathbb{D})}+\overline{(b,z^{n})_{H^{2}(\mathbb{D})}}=\overline{b_{n}}$. Hence, if $a+\overline{b} \perp e_{n}(z)$ and
$a+\overline{b} \perp f_{n}(z)$ for all $n\in \mathbb{N}$ it follows that $a=b=0$.
\end{proof}

From now on, when there is no possibility of confusion, we will not use sub-indices to indicate to what space corresponds the internal product
or the respective norm.

Let $H(\mathbb{D})$ be the space of complex harmonic functions defined on $\mathbb{D}$. There are two main ways to make the composition of two complex harmonic functions
in order to the result be a new complex harmonic function, which maintain the linearity on the left (it will be seen in more detail below)
\begin{equation}\label{s1-2}
(a+\overline{b})\circ_{1} (c+\overline{d})=a\circ c+\overline{b\circ d},\,\,\,\,\,\,\,\,\,\,\,\,\,
(a+\overline{b})\circ_{2} (c+\overline{d})=a\circ d+\overline{b\circ c},
\end{equation}
where $\circ$ denotes the usual composition between analytic functions. From (\ref{s1-2}) follows that the more general composition product of $H(\mathbb{D})\times H(\mathbb{D})$ into $H(\mathbb{D})$
keeping the linearity on the left is the following
\begin{equation}\label{s1-3}
(a+\overline{b})\propto_{(\alpha,\beta,\gamma,\delta)} (c+\overline{d})=\alpha(a\circ c)+\beta(\overline{b\circ d})+\gamma(a\circ d)+\delta(\overline{b\circ c}),
\end{equation}
where $\alpha,\beta,\gamma,\delta\in\mathbb{C}$ are arbitrary but fixed.

We denote $AH(\mathbb{D})=\{\varphi+\overline{\pi}\,\,|\,\,\varphi,\pi\in A(\mathbb{D}):\,\,\mathbb{D}\longrightarrow \mathbb{D}\}$ where $A(\mathbb{D})$ is the space of all analytic functions and
define the composition operator $C^{(\alpha,\beta,\gamma,\delta)}_{\varphi+\overline{\pi}}:\,\,HH^{2}(\mathbb{D})\longrightarrow HH^{2}(\mathbb{D})$ defined of the following form $C^{(\alpha,\beta,\gamma,\delta)}_{\varphi+\overline{\pi}}(a+\overline{b})=(a+\overline{b})\propto_{(\alpha,\beta,\gamma,\delta)} (\varphi+\overline{\pi})=\alpha(a\circ \varphi)+\beta(\overline{b\circ \pi})+
\gamma(a\circ \pi)+\delta(\overline{b\circ \varphi})$ for all $a+\overline{b}\in HH^{2}(\mathbb{D})$.

\begin{remark}The action of $C^{(\alpha,\beta,\gamma,\delta)}_{\varphi+\overline{\pi}}$ can be written in the form
\begin{equation}\label{s1-4}
C^{(\alpha,\beta,\gamma,\delta)}_{\varphi+\overline{\pi}}(a+\overline{b})=\alpha C_{\varphi}a+\beta \overline{C_{\pi}b}+\gamma C_{\pi}a +\delta \overline{C_{\varphi}b}=
\left [\alpha C_{\varphi}a+\gamma C_{\pi}a \right ]+\left [\,\, \overline{\overline{\beta} C_{\pi}b+\overline{\delta} C_{\varphi}b}\,\,\right ]=A^{(\alpha,\beta,\gamma,\delta)}a+
\overline{B^{(\alpha,\beta,\gamma,\delta)}b},
\end{equation}
where $A^{(\alpha,\beta,\gamma,\delta)}=\alpha C_{\varphi}+\gamma C_{\pi}$ and $B^{(\alpha,\beta,\gamma,\delta)}=\overline{\beta}C_{\pi}+\overline{\delta}C_{\varphi}$.
\end{remark}

We have
\begin{lemma}The operators $A^{(\alpha,\beta,\gamma,\delta)}$ and $B^{(\alpha,\beta,\gamma,\delta)}$ defined in the previous remark are bounded linear operators on $H^{2}(\mathbb{D})$.
\end{lemma}
\begin{proof}The linearity of $A^{(\alpha,\beta,\gamma,\delta)}$ and $B^{(\alpha,\beta,\gamma,\delta)}$ follows from the fact that both are linear combinations of
usual bounded linear composition operators whose symbols are analytic functions mapping the unit disk into itself. For the same
reason $A^{(\alpha,\beta,\gamma,\delta)}$ and $B^{(\alpha,\beta,\gamma,\delta)}$ are bounded operators, moreover
\begin{equation}\label{s1-5}
\|A^{(\alpha,\beta,\gamma,\delta)}\|\leq |\alpha|\|C_{\varphi}\|+|\gamma|\|C_{\pi}\|\leq |\alpha|\sqrt{\frac{1+\varphi(0)}{1-\varphi(0)}}
+|\gamma|\sqrt{\frac{1+\pi(0)}{1-\pi(0)}},
\end{equation}
and
\begin{equation}\label{s1-6}
\|B^{(\alpha,\beta,\gamma,\delta)}\|\leq |\delta|\|C_{\varphi}\|+|\beta|\|C_{\pi}\|\leq |\delta|\sqrt{\frac{1+\varphi(0)}{1-\varphi(0)}}
+|\beta|\sqrt{\frac{1+\pi(0)}{1-\pi(0)}},
\end{equation}
\end{proof}

The previous result leads to our next definition

\begin{definition}\label{ds1-1}
Suppose that $A,B\in \mathcal{B}(H^{2}(\mathbb{D}))$ and define $A+\overline{B}: HH^{2}(\mathbb{D})\longrightarrow HH^{2}(\mathbb{D})$
of the following form
\begin{equation*}
(A+\overline{B})(a+\overline{b})=Aa+\overline{Bb},
\end{equation*}
for all $a+\overline{b}\in H^{2}(\mathbb{D})$.
\end{definition}

Then, the following result holds

\begin{theorem}\label{ts1-1}
For all $A,B\in \mathcal{B}(H^{2}(\mathbb{D}))$, we have $A+\overline{B}\in \mathcal{B}(HH^{2}(\mathbb{D}))$.
\end{theorem}
\begin{proof}It is clear that $A+\overline{B}$ is defined in all $HH^{2}(\mathbb{D})$. First, we will prove that $A+\overline{B}$ is linear
\begin{align*}
(A+\overline{B})(\alpha_{1}(a_{1}+\overline{b_{1}})+\alpha_{2}(a_{2}+\overline{b_{2}}))&=(A+\overline{B})\left ((\alpha_{1}a_{1}+\alpha_{2}a_{2})
+\overline{(\overline{\alpha_{1}}b_{1}+\overline{\alpha_{2}}b_{2})}\right )=A(\alpha_{1}a_{1}+\alpha_{2}a_{2})+\overline{B(\overline{\alpha_{1}}b_{1}+\overline{\alpha_{2}}b_{2})} \\
&=(\alpha_{1}Aa_{1}+\alpha_{2}Aa_{2})+\overline{(\overline{\alpha_{1}}Bb_{1}+\overline{\alpha_{2}}Bb_{2})}
=\alpha_{1}(Aa_{1}+\overline{Bb_{1}})+\alpha_{2}(Aa_{2}+\overline{Bb_{2}}) \\
&=\alpha_{1}(A+\overline{B})(a_{1}+\overline{b_{1}})+\alpha_{2}(A+\overline{B})(a_{2}+\overline{b_{2}}),
\end{align*}
thus $A+\overline{B}$ is linear. It remains to be shown that $A+\overline{B}$ is bounded. In fact
\begin{align*}
\|(A+\overline{B})(a+\overline{b})\|^{2}&=\|Aa\|^{2}+\|Bb\|^{2}\leq \|A\|^{2}\|a\|^{2}+\|B\|^{2}\|b\|^{2} \\
&\leq \max\{\|A\|^{2},\|B\|^{2}\}
(\|a\|^{2}+\|b\|^{2})\leq \max\{\|A\|^{2},\|B\|^{2}\}\|(a+\overline{b})\|,
\end{align*}
thus
\begin{equation}\label{s1-7}
\|A+\overline{B}\|\leq \sqrt{\max\{\|A\|^{2},\|B\|^{2}\}}.
\end{equation}
\end{proof}

Notice that from the proof of the theorem follows that $\|A\|,\|B\|\leq \|A+\overline{B}\|$.

\begin{corollary}We have $C^{(\alpha,\beta,\gamma,\delta)}_{\varphi+\overline{\pi}}\in \mathcal{B}(HH^{2}(\mathbb{D}))$.
\end{corollary}

\begin{theorem}All operator $L\in \mathcal{B}(HH^{2}(\mathbb{D}))$ such that $L\,H^{2}(\mathbb{D})\subset H^{2}(\mathbb{D})$ and $L\,\overline{H^{2}(\mathbb{D})}\subset \overline{H^{2}(\mathbb{D})}$
is of the form $L=A_{L}+\overline{B_{L}}$ where this is understood in the sense of the definition \ref{ds1-1}, and $A_{L},B_{L}\in \mathcal{B}(H^{2}(\mathbb{D}))$.
\end{theorem}
\begin{proof}From the linearity of $L$ for every $a+\overline{b}$, we have that $L(a+\overline{b})=La+L\overline{b}$, moreover $La \in H^{2}(\mathbb{D})$ and $L\overline{b}\in \overline{H^{2}(\mathbb{D})}$.
Define now $A_{L}a=La$ and $B_{L}b=\overline{L\overline{b}}$ then $A_{L}$ and $B_{L}$ are linear. Also, it is evident that $D(A_{L})=D(B_{L})=H^{2}(\mathbb{D})$ and
$L(a+\overline{b})=A_{L}a+\overline{B_{L}b}=(A_{L}+\overline{B_{L}})(a+\overline{b})$ in the sense of the definition \ref{ds1-1}. Next, we will see that both operators $A_{L}$ and $B_{L}$ are bounded. Indeed,
\begin{equation*}
\|Aa\|^{2}_{H^{2}(\mathbb{D})}+\|Bb\|^{2}_{H^{2}(\mathbb{D})}=\|Aa+\overline{Bb}\|^{2}_{HH^{2}(\mathbb{D})}=\|(A+\overline{B})
(a+\overline{b})\|^{2}_{HH^{2}(\mathbb{D})}
=\|L(a+\overline{b})\|^{2}_{HH^{2}(\mathbb{D})}\leq \|L\|(\|a\|^{2}_{H^{2}(\mathbb{D})}+\|b\|^{2}_{H^{2}(\mathbb{D})}),
\end{equation*}
it shows that $\|A\|_{\mathcal{B}(H^{2}(\mathbb{D}))},\|B\|_{\mathcal{B}(H^{2}(\mathbb{D}))}\leq \|L\|_{\mathcal{B}(HH^{2}(\mathbb{D}))}$.
\end{proof}

The following lemma will be needed to establish some results of the next subsection.

\begin{lemma}\label{ls1-1}
Suppose that $L=A+\overline{B}\in \mathcal{B}(HH^{2}(\mathbb{D}))$. Then, $L^{\ast}=A^{\ast}+\overline{B^{\ast}}$ and moreover $L^{\ast}\in \mathcal{B}(HH^{2}(\mathbb{D}))$.
\end{lemma}
\begin{proof}For any $f+\overline{g},l+\overline{m}\in HH^{2}(\mathbb{D})$ one has
\begin{align*}
\left ((A+\overline{B})(f+\overline{g}) ,l+\overline{m} \right )&=\left (Af+\overline{Bg},l+\overline{m} \right )=(Af,l)+\overline{(Bg,m)}=(f,A^{\ast}l)+\overline{(g,B^{\ast}m)}   \\
&=(f,A^{\ast}l)+\overline{(g,B^{\ast}m)}=(f+\overline{g}, A^{\ast}l+B^{\ast}m)=(f+\overline{g},\left (A^{\ast}+\overline{B^{\ast}}\,\right )(l+\overline{m})),
\end{align*}
now, that $A^{\ast}+\overline{B^{\ast}}$ is bounded follows from theorem \ref{ts1-1}.
\end{proof}

Consider now $\mathcal{E}=\mathcal{B}(H^{2}(\mathbb{D}))+\overline{\mathcal{B}(H^{2}(\mathbb{D}))}\subset \mathcal{B}(HH^{2}(\mathbb{D}))$. The space $\mathcal{E}$ may be regarded a vector space
if one defines $(A+\overline{B})+(C+\overline{D})=(A+C)+\overline{(B+D)}$ and $\lambda (A+\overline{B})=\lambda A+\overline{\lambda B}$ where $\lambda\in \mathbb{C}$. Observe
that if $A+\overline{B}\in \mathcal{E}$ then $(A+\overline{B})^{\ast}=A^{\ast}+\overline{B^{\ast}}\in \mathcal{E}$. The composition of operators of this type leads to the equality $(A+\overline{B})(C+\overline{D})
=AC+\overline{BD}\in \mathcal{E}$. Next, we intend to show that $\mathcal{E}$ is a closed subspace of $\mathcal{B}(HH^{2}(\mathbb{D}))$, suppose that $A_{n}+\overline{B}_{n}\longrightarrow L$ in
$\mathcal{B}(HH^{2}(\mathbb{D}))$, this means that $\{A_{n}+\overline{B}_{n}\}$ is a Cauchy sequence then so are $\{A_{n}\}$ and $\{B_{n}\}$. Taking into account that $\mathcal{B}(H^{2}(\mathbb{D}))$
is complete it follows that there are $A, B\in \mathcal{B}(H^{2}(\mathbb{D}))$ such that $A_{n}\longrightarrow A$ and $B_{n}\longrightarrow B$ which implies that $A_{n}+\overline{B}_{n}\longrightarrow
A+\overline{B}$, thus $L=A+\overline{B}$.

\subsection{Reproducing kernel and some composition operators}

In this part of our work, we will characterize the composition operators introduced above, in terms of complex harmonic reproduction kernels.

\begin{definition}We shall say that $k(\lambda,z)+\overline{i(\lambda,z)}$ is a reproducing kernel in $HH^{2}(\mathbb{D})$ if
$$(a(z)+\overline{b(z)},k(\lambda,z)+\overline{i(\lambda,z)})=a(\lambda)+\overline{b(\lambda)},$$
for all $a+\overline{b}\in HH^{2}(\mathbb{D})$ and for all $\lambda\in \mathbb{D}$.
It is clear that in this case both $k(\lambda,z)$ and $i(\lambda,z)$ must be reproducing kernels in $H^{2}(\mathbb{D})$.
\end{definition}

Below, by a simple composition operator, we mean an operator of the form $C^{(1,1,0,0)}_{\varphi+\overline{\pi}}=C_{\varphi+\overline{\pi}}=C_{\varphi}+\overline{C_{\pi}}$.
We recall that $\varphi, \pi : \mathbb{D}\longrightarrow \mathbb{D}$ are analytic functions.

\begin{proposition}For all reproducing kernel $k(\lambda,z)+\overline{i(\lambda,z)}$ we have
\begin{equation}\label{s1-8}
(C_{\varphi+\overline{\pi}})^{\ast}(k(\lambda,z)+\overline{i(\lambda,z)})=k(\varphi(\lambda),z)+\overline{i(\pi(\lambda),z)}.
\end{equation}
\end{proposition}
\begin{proof}For every $a+\overline{b}\in HH^{2}(\mathbb{D})$ we obtain
\begin{align*}
\left (a+\overline{b},(C_{\varphi+\overline{\pi}})^{\ast}\left (k(\lambda,z)+\overline{i(\lambda,z)}\right )\right )_{HH^{2}(\mathbb{D})}&=
\left (C_{\varphi+\overline{\pi}}(a+\overline{b}), k(\lambda,z)+\overline{i(\lambda,z)} \right )_{HH^{2}(\mathbb{D})} \\
&=\left (C_{\varphi}a+\overline{C_{\pi} b}, k(\lambda,z)+\overline{i(\lambda,z)} \right )_{HH^{2}(\mathbb{D})} \\
&=(C_{\varphi}a, k(\lambda,z))_{H^{2}(\mathbb{D})}+\overline{(C_{\pi}b, i(\lambda,z))_{H^{2}(\mathbb{D})}} \\
&=((a\circ \varphi)(z), k(\lambda,z))_{H^{2}(\mathbb{D})}+\overline{((b\circ \pi)(z), i(\lambda,z))_{H^{2}(\mathbb{D})}} \\
&=a(\varphi(\lambda))+\overline{b(\pi(\lambda))}=a\circ\varphi (\lambda)+\overline{b\circ\pi (\lambda)}.
\end{align*}

But also,
\begin{equation*}
\left (a+\overline{b},k(\varphi(\lambda),z)+\overline{i(\pi(\lambda),z)} \right )_{H^{2}(\mathbb{D})}=(a, k(\varphi(\lambda),z))_{H^{2}(\mathbb{D})}+
\overline{(b, i(\varphi(\lambda),z))_{H^{2}(\mathbb{D})}}=a(\varphi(\lambda))+\overline{b(\pi(\lambda))}=a\circ\varphi (\lambda)+\overline{b\circ\pi (\lambda)},
\end{equation*}
and since $a+\overline{b}$ is arbitrary, (\ref{s1-8}) holds.
\end{proof}

\begin{theorem}The operator $L=A+\overline{B}$ where $A,B\in \mathcal{B}(H^{2}(\mathbb{D}))$ is a simple composition operator in $HH^{2}(\mathbb{D})$
if and only if $A^{\ast}$ and $B^{\ast}$ map the space of reproducing kernels of $H^{2}(\mathbb{D})$ in itself.
\end{theorem}
\begin{proof}We already have seen that (see (\ref{s1-8})) $C^{\ast}_{\varphi+\overline{\pi}}(k(\lambda,z)+\overline{i(\lambda,z)})=k(\varphi(\lambda),z)+\overline{i(\pi(\lambda),z)}$. Moreover, from the
lemma \ref{ls1-1}, we conclude that $C^{\ast}_{\varphi+\overline{\pi}}=C^{\ast}_{\varphi}+\overline{C^{\ast}_{\pi}}$ hence
$C^{\ast}_{\varphi}k(\lambda,z)=k(\varphi(\lambda),z)$ and $C^{\ast}_{\pi}i(\lambda,z)=i(\pi(\lambda),z)$.

Reciprocally, suppose $(A^{\ast}+\overline{B^{\ast}})(k(\lambda,z)+\overline{i(\lambda,z)})=A^{\ast}k(\lambda,z)+\overline{B^{\ast}i(\lambda,z)}
=k(\lambda^{'},z)+\overline{i(\lambda^{''},z)}$, where $k(\lambda^{'},z)$ and $i(\lambda^{''},z)$ are reproducing kernels in $H^{2}(\mathbb{D})$.
It is clear that $A^{\ast}k(\lambda,z)=k(\lambda^{'},z)$ and $B^{\ast}i(\lambda,z)=i(\lambda^{''},z)$.
Define $\varphi, \pi :\mathbb{D}\longrightarrow \mathbb{D}$ in the following form
$\lambda^{'}=\varphi(\lambda)$ and $\lambda^{''}=\pi(\lambda)$, Observe that it possible because $\lambda^{'}, \lambda^{''}\in \mathbb{D}$.
Then, $\forall f \in H^{2}(\mathbb{D})$
\begin{equation*}
(Af)(\lambda)=(Af,k(\lambda,z))=(f,A^{\ast}k(\lambda,z))=(f,k(\lambda^{'},z))=(f,k(\varphi(\lambda),z))=f(\varphi(\lambda)),
\end{equation*}
and $\forall g \in H^{2}(\mathbb{D})$
\begin{equation*}
(Bg)(\lambda)=(Bg,i(\lambda,z))=(g,B^{\ast}i(\lambda,z))=(g,i(\lambda^{''},z))=(g,k(\pi(\lambda),z))=g(\pi(\lambda)).
\end{equation*}

Taking into account that $A, B : H^{2}(\mathbb{D})\longrightarrow H^{2}(\mathbb{D})$ if $f=g=z$ then $(Az)(\lambda)=
\varphi(\lambda)\in H^{2}(\mathbb{D})$ and $(Bz)(\lambda)=\pi(\lambda)\in H^{2}(\mathbb{D})$. Since $f,g \in H^{2}(\mathbb{D})$
are arbitrary, we obtain $A=C_{\varphi}$ and $B=C_{\pi}$. It follows that $L=C_{\varphi}+\overline{C_{\pi}}$.
\end{proof}

\begin{theorem}$L=A+\overline{B}\in\mathcal{E}$ is a simple composition operator in
$HH^{2}(\mathbb{D})$ if and only if $Le_{n}=Ae_{n}=(Ae_{1})^{n}=(L e_{1})^{n}$ and $Lf_{n}=\overline{Be_{n}}=\overline{(Be_{1})^{n}}
=\left (\,\overline{Be_{1}}\, \right )^{n}=(L f_{1})^{n}$ for all $n\in \mathbb{N}$, that is, if and only if $Ae_{n}=(Ae_{1})^{n}$
and $Be_{n}=(Be_{1})^{n}$ for any $n\in \mathbb{N}$.
\end{theorem}
\begin{proof}Suppose first that $L=C_{\varphi}+\overline{C_{\pi}}$ then $Le_{n}=(C_{\varphi}+\overline{C_{\pi}})e_{n}=C_{\varphi}z^{n}
=\varphi^{n}(z)=(C_{\varphi}z)^{n}=(Le_{1})^{n}$ for all $n\in \mathbb{N}$. In the same way, $\forall n\in \mathbb{N}$ we have $Lf_{n}=(C_{\varphi}+\overline{C_{\pi}})f_{n}=\overline{C_{\pi}z^{n}}=\overline{\pi^{n}(z)}=\left (\overline{\pi(z)}\right )^{n}
=\left (\,\overline{C_{\pi}z}\,\right )^{n}=\left (Lf_{1} \right )^{n}$.

Conversely, assume that we have $L=A+\overline{B}$ where $A,B\in \mathcal{B}(H^{2}(\mathbb{D}))$ such that $Az^{n}=(Az)^{n}$ and
$Bz^{n}=(Bz)^{n}$. We must show that $L$ is a simple composition operator. Define $\varphi(z)=Az$ and $\pi(z)=Bz$, then
$\varphi,\pi\in H^{2}(\mathbb{D})$. Now, $Le_{n}=Lz^{n}=Az^{n}=(Az)^{n}=\varphi^{n}(z)$ and $Lf_{n}=L\overline{z^{n}}=\overline{Bz^{n}}=\overline{(Bz)^{n}}=\overline{\pi^{n}(z)}$. It implies that if
$l+\overline{m}=\Sigma\, l_{n}e_{n}+\Sigma\, \overline{m_{n}e_{n}}$ hence $L(l+\overline{m})=(A+\overline{B})(l+\overline{m})=\Sigma\, l_{n}Ae_{n}+
\Sigma\, \overline{m_{n}Be_{n}}=\Sigma\, l_{n}\varphi^{n}(z)+
\Sigma\, \overline{m_{n}\pi^{n}(z)}=(C_{\varphi}+\overline{C_{\pi}})(l+\overline{m})$ and since $l+\overline{m}\in HH^{2}(\mathbb{D})$ is arbitrary,
in order to see that $L=A+\overline{B}$ is a simple composition operator is sufficient to prove that both functions $\varphi, \pi$ map $\mathbb{D}$ into
$\mathbb{D}$. For this purpose it is sufficient to point out that $(\varphi(z))^{n}=Ae_{n}$ and $(\pi(z))^{n}=Be_{n}$, hence
$\|(\varphi(z))^{n}\|\leq \|A\|$ and $\|(\pi(z))^{n}\|\leq \|B\|$ (see \cite{shapiro} page 169).
\end{proof}

For $l+\overline{m},p+\overline{q}\in HH^{2}(\mathbb{D})$, we define the product $(l+\overline{m})(p+\overline{q})=lp+\overline{mq}$.

\begin{theorem}An operator $L=A+\overline{B}$ where $A,B\in \mathcal{B}(H^{2}(\mathbb{D}))$ is a simple composition operator if and only if
\begin{equation}\label{s1-9}
L\left ((l+\overline{m})(p+\overline{q}) \right )= \left (L(l+\overline{m})L(p+\overline{q}) \right ),
\end{equation}
\end{theorem}
\begin{proof}Note that if $L=A+\overline{B}$ then $L\left ((l+\overline{m})(p+\overline{q}) \right )=L(lp+\overline{mq})=A(lp)+\overline{B(mq)}$.
Hence, if $L=C_{\varphi}+\overline{C_{\pi}}$ where $\varphi,\pi : \mathbb{D}\longrightarrow \mathbb{D}$ are analytic functions, then
\begin{align*}
(C_{\varphi}+\overline{C_{\pi}})\left ((l+\overline{m})(p+\overline{q}) \right )&=C_{\varphi}(lp)+\overline{C_{\pi}(mq)}
=(C_{\varphi}l)\,(C_{\varphi}p)+\overline{(C_{\pi}m)\,(C_{\pi}q)} \\
&=(C_{\varphi}l+\overline{C_{\pi}m}\,)(C_{\varphi}p+\overline{C_{\pi}q}\,)=\left ((C_{\varphi}+\overline{C_{\pi}}\,)(l+\overline{m})\right )
\left ((C_{\varphi}+\overline{C_{\pi}}\,)(p+\overline{q})\right ).
\end{align*}

On the other hand, suppose that $L=A+\overline{B}$ satisfies (\ref{s1-9}), then $Le_{n}=Lz^{n}=\underbrace{Lz\cdots Lz}_{n-factors}=(Le_{1})^{n}$ and
$Lf_{n}=L\overline{z^{n}}=L(\overline{z})^{n}=\underbrace{L\overline{z}\cdots L\overline{z}}_{n-factors}=(Lf_{1})^{n}$. So, from the previous lemma
follows that $L$ is a simple composition operator.
\end{proof}

\begin{remark}If $L=A+\overline{B}\in B(HH^{2}(\mathbb{D}))$ it has been pointed out above that $L^{\ast}=A^{\ast}+\overline{B^{\ast}}$. Then
\begin{equation*}
LL^{\ast}(l+\overline{m})=(A+\overline{B})(A^{\ast}+\overline{B^{\ast}})(l+\overline{m})=(A+\overline{B})(A^{\ast}l+\overline{B^{\ast}m})=
(AA^{\ast}l+\overline{BB^{\ast}m}),
\end{equation*}
on the other hand, $L^{\ast}L(l+\overline{m})=A^{\ast}Al+\overline{B^{\ast}Bm}$. It shows that $(LL^{\ast}-L^{\ast}L)e_{n}=(AA^{\ast}-A^{\ast}A)e_{n}$
and $(LL^{\ast}-L^{\ast}L)f_{n}=\overline{(BB^{\ast}-B^{\ast}B)e_{n}}$. Hence, $L$ is a normal operator if and only if $A$ and $B$ are normal operators.
In the particular case in which $L=C_{\varphi+\overline{\pi}}=C_{\varphi}+\overline{C_{\pi}}$, it will be a normal operator if and only if $C_{\varphi}$
and $C_{\pi}$ are normal operators, that is, if and only if $\varphi(z)=\lambda z$ and $\pi(z)=\mu z$ where $|\lambda|\leq 1$ and $|\mu|\leq 1$.
\end{remark}

\section*{Acknowledgment}

L. E. Ben\'{\i}tez-Babilonia was support in part by the Cordoba University and in part by CONACYT grant $45886$. He also thanks the hospitality of CIMAT during
his visit to the Center for some days in June $2019$, period in which this work was completed. The second named author was supported under CONACYT grant $45886$.

\end{document}